\documentclass[12pt]{article}
\usepackage{graphicx}
\usepackage{amsmath,amsfonts,amssymb,amsthm}
\usepackage{bbm,mathrsfs,bm,mathtools}
\title{Pseudo-cones}
\author{Rolf Schneider}
\date{}
\date{}
\newcommand{\R}{{\mathbb R}}
\newcommand{\C}{{\mathcal C}}
\newcommand{\K}{{\mathcal K}}

\newcommand{\N}{{\mathbb N}}

\newcommand{\B}{\mathcal{B}}
\newcommand{\F}{{\mathcal F}}

  \renewcommand{\dim}{{\rm dim}\,}

  \newcommand{\cl}{{\rm cl}\,}

  \newcommand{\D}{{\rm d}}

  \newcommand{\fed}{\,\rule{.1mm}{.20cm}\rule{.20cm}{.1mm}\,}

\newtheorem{theorem}{Theorem}
\newtheorem{lemma}{Lemma}
\newtheorem{proposition}{Proposition}
\newtheorem{definition}{Definition}

\begin{document}
\maketitle

\begin{abstract}
Pseudo-cones are a class of unbounded closed convex sets, not containing the origin. They admit a kind of polarity, called copolarity. With this, they can be considered as a counterpart to convex bodies containing the origin in the interior. The purpose of the following is to study this analogy in greater detail. We supplement the investigation of copolarity, considering, for example, conjugate faces. Then we deal with the question suggested by Minkowski's theorem, asking which measures are surface area measures of pseudo-cones with given recession cone. We provide a sufficient condition for possibly infinite measures and a special class of pseudo-cones.
\\[2mm]
\end{abstract}

\section{Introduction}\label{sec1}

We work in $\R^n$ ($n\ge 2$), the $n$-dimensional Euclidean vector space with origin $o$, scalar product $\langle\cdot\,,\cdot\rangle$ and norm $\|\cdot\|$.

Essentially following \cite{XLL22}, we say that a subset $K\subset\R^n$  is a {\em pseudo-cone} if it is a nonempty closed convex set not containing the origin and satisfying $\lambda x\in K$ for $x\in K$ and $\lambda\ge 1$ (thus, we include closedness in the definition). Equivalently, a pseudo-cone is the nonempty intersection of a family of closed halfspaces not containing the origin in the interior, where at least one of the halfspaces does not contain the origin. 
Using the terminology of \cite{Ras17} and \cite{XLL22}, we define the {\em copolar set} of a pseudo-cone $K$ by
$$  K^\star \coloneqq  \{x\in\R^n: \langle x,y\rangle\le -1 \mbox{ for all }y\in K\}.$$
Obviously, this is again a pseudo-cone, and one can show that $K^{\star\star}\coloneqq  (K^\star)^\star=K$ (see, for example, \cite{ASW22}). 

Pseudo-cones can be considered as a counterpart to convex bodies containing the origin in the interior, and copolarity plays a similar role for pseudo-cones as the ordinary polarity does for convex bodies. It is the purpose of Sections \ref{sec3} and \ref{sec4} to make this evident, and to point out some differences.

If $K$ is a pseudo-cone and $C$ is its recession cone (always assumed to be pointed and full-dimensional), we say that $K$ is a $C$-pseudo-cone. Note that if $K$ is a $C$-pseudo-cone, then $K\subset C$ (see \cite{XLL22} for the reverse). Several types of $C$-pseudo-cones have been distinguished. $K$ is called {\em $C$-full} if $C\setminus K$ is bounded, and {\em $C$-close} if $C\setminus K$ has finite volume. The set $K$ is called {\em $C$-asymptotic} if the distance of $x\in\partial K$ from $\partial C$ tends to zero as $\|x\|\to\infty$. If $K$ is a $C$-pseudo-cone and if $z\in C$, then also $K+z$ is a $C$-pseudo-cone. Thus, a $C$-pseudo-cone need not be $C$-asymptotic.

For $C$-coconvex sets, that is, sets $C\setminus K$ where $K$ is a $C$-close pseudo-cone, a Brunn--Minkowski theory was developed in \cite{Sch18, Sch21} (and subsumed in \cite[Chap. 7]{Sch22}). As shown by Yang, Ye and  Zhu \cite{YYZ22}, also parts of the $L_p$ Brunn--Minkowski theory can be carried over. 

In Section \ref{sec6}, we deal again with the Brunn--Minkowski theory of $C$-close sets, which are a particular class of pseudo-cones. Minkowski's classical existence theorem (see, e.g., \cite[Sec. 8]{Sch14}) answers the following question: What are the necessary and sufficient conditions on a Borel measure on the unit sphere to be the surface area measure of a convex body? This question can also be formulated for pseudo-cones. To explain this, let $K$ be a pseudo-cone in $\R^n$ with recession cone $C$. Let $\Omega_{C^\circ}\coloneqq  S^{n-1}\cap{\rm int}\,C^\circ$ (this notation differs from the one used in \cite{Sch18}), where 
$$ C^\circ  \coloneqq  \{x\in\R^n: \langle x,y\rangle\le 0 \mbox{ for all }y\in C\}$$
denotes the polar cone of $C$ and $S^{n-1}$ is the unit sphere. For each $u\in \Omega_{C^\circ}$, the closed convex set $K$ has a supporting hyperplane with outer normal vector $u$. For a Borel set $\omega\subset\Omega_{C^\circ}$, the reverse spherical image ${\bm x}_K(\omega)$ of $K$ at $\omega$ is defined (as for convex bodies, but with different notation, cf. \cite[p. 88]{Sch14}) as the set of all points $x\in\partial K$ at which there exists a supporting hyperplane with outer unit normal vector belonging to $\omega$. Then one defines
$$ S_{n-1}(K,\omega)\coloneqq  {\mathscr H}^{n-1}({\bm x}_K(\omega)),$$
where $ {\mathscr H}^{n-1}$ is the $(n-1)$-dimensional Hausdorff measure. This yields a Borel measure $S_{n-1}(K,\cdot)$ on $\Omega_{C^\circ}$, the {\em surface area measure} of $K$. Thus, the surface area measure of a pseudo-cone is only defined on an open proper subset of the unit sphere (in fact, of a hemisphere), and it may be infinite. One may now ask for necessary and sufficient conditions on a Borel measure on $\Omega_{C^\circ}$ to be the surface area measure of a $C$-pseudo-cone. It was shown in \cite[Thm. 3]{Sch18} that every nonzero finite Borel measure on $\Omega_{C^\circ}$ with compact support (contained in $\Omega_{C^\circ}$) is the surface area measure of a $C$-full set.  In \cite[Thm. 1]{Sch21} it was then proved that every nonzero finite Borel measure on $\Omega_{C^\circ}$ is the surface area measure of a $C$-close set. This set is uniquely determined. In \cite{YYZ22}, these results were carried over to $L_p$ surface area measures,

The mentioned existence results are restricted to finite measures. If infinite measures are allowed, the situation changes. Not every infinite Borel measure on $\Omega_{C^\circ}$ is the surface area measure of some pseudo-cone. Local finiteness (that is, finiteness on compact subsets of $\Omega_{C^\circ}$) is necessary, but not sufficient. A non-trivial necessary condition for surface area measures of $C$-pseudo-cones was found in \cite[Sec. 4]{Sch21}. The subsequent theorem (proved in Section \ref{sec6}, after some preparations in Section 5) provides a sufficient condition for possibly infinite measures to be the surface area measure of a $C$-close set. For this, we need the following definition.

\begin{definition}\label{D1.1}
For $u\in\Omega_{C^\circ}$, let $\delta_C(u)$ be the spherical distance of $u$ from the boundary $\partial\Omega_{C^\circ}$ of $\Omega_{C^\circ}$, that is,
$$ \delta_C(u)\coloneqq  \min\{\angle(u,v): v\in\partial\Omega_{C^\circ}\},$$
where $\angle(u,v)$ is the geodesic distance between $u$ and $v$ (the angle between the vectors $u$ and $v$). 
\end{definition}
 
The condition in the theorem below ensures moderate growth of a measure on $\Omega_{C^\circ}$ when approaching the boundary of $\Omega_{C^\circ}$.

\begin{theorem}\label{T1.1} 
Let $\varphi$ be a nonzero Borel measure on $\Omega_{C^\circ}$. If 
$$\int_{\Omega_{C^\circ}} \delta_C^{1/n}\,\D\varphi<\infty,$$ 
then $\varphi$ is the surface area measure of a $C$-close pseudo-cone.
\end{theorem}

However, the question for a necessary and sufficient condition remains open. We recall that the uniqueness was already settled in \cite[Thm. 2]{Sch18}: If $K_0,K_1$ are $C$-close sets with $S_{n-1}(K_0,\cdot)=S_{n-1}(K_1,\cdot)$, then $K_0=K_1$.


\section{Preliminaries on pseudo-cones}\label{sec2}

We fix some notation and collect some facts about pseudo-cones, which either will be used below or are of independent interest. The set of all pseudo-cones in $\R^n$ is denoted by $ps\C^n$.

As usual, ${\rm int}\,S$ and ${\rm cl}\,S$ stand for the interior and closure of a set $S$. The boundary of $S$ is denoted by $\partial S$, in $\R^n$ as well as in $S^{n-1}$. The convex hull of a set $A\subset\R^n$ is ${\rm conv}\,A$. The $k$-dimensional volume, where it exists, is denoted by $V_k$. By $B^n$ we denote the unit ball of $\R^n$ with center at the origin. Hyperplanes and closed halfspaces are written in the following form. For $u\in \R^n\setminus\{o\}$  and $t\in\R$,
$$ H(u,t)\coloneqq  \{x\in\R^n: \langle x,u\rangle = t\}, \qquad H^-(u,t)\coloneqq  \{x\in\R^n: \langle x,u\rangle \le t\}.$$
If a closed convex set $K$ has a supporting halfspace with outer normal vector $u\in S^{n-1}$, then this halfspace is denoted by $H^-(K,u)$, and its boundary by $H(K,u)$.

The {\em recession cone} of a nonempty closed convex set $K\subset \R^n$ can be defined by
$$ {\rm rec}\,K\coloneqq  \{y\in\R^n: x+\lambda y\in K \mbox{ for all }x\in K\mbox{ and all }\lambda\ge 0\},$$
from which it follows that a pseudo-cone $K$ satisfies $K\subset {\rm rec}\,K$.

Throughout the following, $C$ is a fixed closed convex cone, assumed to be pointed and with nonempty interior. We can choose a unit vector $\mathfrak{v}\in{\rm int}\,C$ such that also $-\mathfrak{v}\in {\rm int}\,C^\circ$. In fact, if ${\rm int}\,C\cap{\rm int}\,(-C^\circ)=\emptyset$, then $C$ and $-C^\circ$ can be separated by a hyperplane (see \cite[Thm. 1.3.8]{Sch14}). Then $C$ and $C^\circ$ are contained in the same closed halfspace, a contradiction. Also $\mathfrak{v}$ is fixed in the following. For $t>0$, we write
$$ C(t)\coloneqq  C\cap H(\mathfrak{v},t),\qquad C^-(t)\coloneqq  C\cap H^-(\mathfrak{v},t).$$
These sets are nonempty and compact, due to the choice of $\mathfrak{v}$. The definition
$$ \Omega_{C^\circ} = S^{n-1} \cap {\rm int}\,C^\circ $$
was already mentioned. The notation $\B(\omega)$ is used for the $\sigma$-algebra of all Borel subsets of an open or compact subset $\omega\subseteq S^{n-1}$.

For a pseudo-cone $K$ with recession cone $C$, we define the {\em support function} $h(K,\cdot):C^\circ\to\R$ by
$$ h(K,u)\coloneqq  \max\{\langle u,x\rangle: x\in K\}\quad\mbox{for }u\in{\rm int}\,C^\circ$$
and
$$ h(K,u)\coloneqq  \sup\{\langle u,x\rangle: x\in K\}\quad\mbox{for }u\in\partial C^\circ.$$
Then $-\infty<h(K,u)\le 0$ for $u\in C^\circ$ and $h(K,u)<0$ for $u\in{\rm int}\,C^\circ$. It is clear that for $u\in{\rm int}\,C^\circ$ the maximum exists and is negative, since $K\subset C$ and $o\notin K$. To avoid many minus signs, we write $\overline h(K,\cdot) \coloneqq  -h(K,\cdot)$, so that
$$ \overline h(K,u) = \min\{|\langle x,u\rangle|: x\in K\}\quad\mbox{for }u\in{\rm int}\,C^\circ.$$
Instead of $\overline h(K,\cdot)$ we also write $\overline h_K$, whenever this is more convenient. We observe that $\overline h_K\le \overline h_L$ for $C$-pseudo-cones $K$ and $L$ implies $K\supseteq L$.

\begin{definition}\label{D2.1}
For a pseudo-cone $K$, we denote by $b(K)$ its distance from $o$, that is, 
$$ b(K)\coloneqq  \min\{r>0: rB^n\cap K\not=\emptyset\}.$$
\end{definition}

Since any supporting hyperplane of a pseudo-cone $K$ must intersect the ball $b(K)B^n$ (otherwise, this ball could be increased without intersecting $K$), the support function of $K$ satisfies
\begin{equation}\label{2.1} \overline h_K\le b(K).
\end{equation}

We expand Definition 1 from \cite{Sch18}.
\begin{definition}\label{D2.2}
If $K_j$, $j\in\N_0$, are $C$-pseudo-cones, we write
$$ K_j\to K_0$$
and say that $(K_j)_{j\in\N}$ converges to $K_0$ if there exists $t_0>0$ such that  $K_j\cap C^-(t_0)\not=\emptyset$ for all $j\in \N_0$ and
$$ \lim_{j\to\infty} (K_j\cap C^-(t)) = K_0\cap C^-(t)\quad\mbox{for all }t\ge t_0,$$
where the latter means convergence of convex bodies with respect to the Hausdorff metric.
\end{definition}

The following theorem is a counterpart to the Blaschke selection theorem for convex bodies.

\begin{lemma}\label{L2.2}{\em [Selection theorem for $C$-pseudo-cones]}
Every sequence of $C$-pseudo-cones whose distances from the origin are bounded and bounded away from $0$ has a subsequence that converges to a $C$-pseudo-cone.
\end{lemma}

\begin{proof}
Let $(K(j))_{j\in\N}$ be a sequence of $C$-pseudo-cones whose distances from the origin are bounded and bounded away from $0$. Then there is a constant $t_1>0$ such that $K(j)\cap C^-(t_1)\not=\emptyset$ for all $j\in\N$. Choose a sequence $t_1<t_2 <t_3< \dots$ tending to infinity. By the Blaschke selection theorem (see, e.g., \cite[Thm. 1.8.7]{Sch14}), the bounded sequence $(K(j)\cap C^-(t_1))_{j\in\N}$ of nonempty convex bodies has a convergent subsequence. Hence, there is a subsequence  $(j^{(1)}_i)_{i\in \N}$ of $(j)_{j\in\N}$ such that
$$ \lim_{i\to\infty} K(j_i^{(1)})\cap C^-(t_1)  = M_1$$ 
for some convex body $M_1$. We have $o\notin M_1$, since for the sequence $(K(j))_{j\in\N}$ the distances from the origin are bounded away from zero. Similarly, there is a subsequence $(j_i^{(2)})_{i\in\N}$ of $(j_i^{(1)})_{i\in\N}$ such that
$$ \lim_{i\to\infty} K(j_i^{(2)})\cap C^-(t_2)  = M_2$$ 
for some convex body $M_2$. In particular, the latter implies that
$$ M_2\cap C^-(t_1) =\lim_{i\to\infty} (K(j_i^{(2)})\cap C^-(t_2))\cap C^-(t_1) =M_1.$$
By induction, we obtain for each $k\ge 2$ a subsequence $(j_i^{(k)})_{i\in\N}$ of $(j_i^{(k-1)})_{i\in\N}$ such that
$$
\lim_{i\to\infty} K(j_i^{(k)})\cap C^-(t_k) = M_k
$$
for some convex body $M_k$ satisfying
$$ M_k\cap C^-(t_{k-1}) = M_{k-1}.$$
The diagonal sequence $(j_i)_{i\in\N}\coloneqq (j_i^{(i)})_{i\in\N}$ is a subsequence of each sequence $(j_i^{(k)})_{i\in\N}$, hence
$$ \lim_{i\to\infty} K(j_i)\cap C^-(t_k)= M_k$$
for each $k\in \N$. 

Now we define 
$$ M\coloneqq  \bigcup_{k\in\N} M_k.$$
Then
$$ M\cap C^-(t_k)= M_k$$ 
for $k\in\N$. The latter implies that $M$ is a closed convex set. 

We show that $M$ is a pseudo-cone. We have $o\notin M$ since $o\notin M_1$. Let $x\in M$ and $\lambda\ge 1$. Choose $k$ such that $\lambda x\in {\rm int}\,C^-(t_k)$. Since $\lim_{i\to\infty} K(j_i)\cap C^-(t_k)=M_k$, there is (by \cite[Thm. 1.3.8]{Sch14}) a sequence $(x_i)_{i\in\N}$ with $x_i\in K(j_i)\cap C^-(t_k)$ and $x_i\to x$ as $i\to\infty$. Since $\lambda x_i\in K(j_i)\cap C^-(t_k)$ for sufficiently large $i$ (since $\lambda x\in{\rm int}\,C^-(t_k)$), we deduce that $\lambda x\in M_k\subset M$. Thus $M$ is a pseudo-cone. In a similar way, one shows that $y\in C$ and $x\in M$ implies that $x+y\in M$, hence $y\in {\rm rec}\,M$ and thus $C\subseteq {\rm rec}\,M$. Conversely, suppose that $y\in\R^n\setminus C$. Then there is a vector $u$ such that $\langle x,u\rangle\le 0$ for $x\in C$ and $\langle y,u\rangle>0$. Let $x\in C$ and $\lambda>0$. We have $\langle x+\lambda y,u\rangle >0$ for large $\lambda$ and hence $x+\lambda y\notin C$, which means that $x+\lambda y\notin M$. Thus, $y\notin{\rm rec}\,M$. We have shown that ${\rm rec}\,M=C$. Thus, $M$ is a pseudo-cone with recession cone $C$. 

From  
$$ \lim_{i\to\infty} K(j_i)\cap C^-(t_k)= M\cap C^-(t_k)\quad\mbox{for }k\in\N$$
and the definition of the convergence of $C$-pseudo-cones it follows that $K(j_i)\to M$ as $i\to\infty$.
\end{proof}

The following shows that the function $K\mapsto V_n(C\setminus K)$ on $C$-pseudo-cones is lower semi-continuous, but not continuous.

\begin{lemma}\label{L2.3}
Let $K_i$, $i\in\N_0$, be $C$-pseudo-cones such that $K_i\to K_0$ as $i\to\infty$. Then
$$ V_n(C\setminus K_0) \le \liminf_{i\to\infty} V_n(C\setminus K_i).$$
Here strict inequality can hold.
\end{lemma}

\begin{proof}
First we assume that $ V_n(C\setminus K_0)=\infty$. Let $a>0$. There exists $t_0$ such that $V_n(C^-(t_0)\setminus K_0)>a$. Since $V_n(C^-(t_0)\cap K_i)\to V_n(C^-(t_0)\cap K_0)$ as $i\to\infty$, there is a number $i_0$ with $V_n(C^-(t_0)\setminus K_i)>a/2$ for $i\ge i_0$, hence $V_n(C\setminus K_i)>a/2$ for $i\ge i_0$. Since $a>0$ was arbitrary, it follows that $\liminf_{i\to\infty} V_n(C\setminus K_i)=\infty$. 

Now we assume that $V_n(C\setminus K_0)<\infty$. Let $\varepsilon>0$. We can choose $t>0$ with $V_n(C^-(t)\setminus K_0) \ge V_n(C \setminus K_0) -\varepsilon$. Since $K_i\cap C^-(t) \to K_0\cap C^-(t)$ as $i\to\infty$, for all sufficiently large $i$ we have
$$ V_n(C\setminus K_i) \ge V_n(C^-(t)\setminus K_i) \ge V_n(C^-(t)\setminus K_0)-\epsilon\ge V_n(C\setminus K_0)-2\varepsilon,$$
thus 
$$ V_n(C\setminus K_0)\le \liminf_{i\to\infty} V_n(C\setminus K_i) +2\varepsilon.$$
Since $\varepsilon>0$ was arbitrary, it follows that $V_n(C\setminus K_0)\le \liminf_{i\to\infty} V_n(C\setminus K_i)$. 

To show that strict inequality is possible, we choose a $C$-close pseudo-cone $K_0\subset{\rm int}\,C$ and a sequence $t_1< t_2 < t_3 <\dots \to\infty$ with $K_0\cap C^-(t_1)\not=\emptyset$. For $i\in\N$ we define
$$ K_i\coloneqq  (K_0\cap C^-(t_i))+C,$$
which is a $C$-pseudo-cone. Since $K_0\cap C^-(t_i)\subset {\rm int}\,C$, it easy to see that $V_n(C\setminus K_i) =\infty$. On the other hand, $\lim_{i\to\infty} K_i=K_0$ and $V_n(C\setminus K_0)<\infty$.

The example can be modified so that $\liminf_{i\to\infty} V_n(C\setminus K_i)$ is a finite value larger than $V_n(C\setminus K_0)$.
\end{proof}

\section{Copolarity}\label{sec3}

For an arbitrary set $\emptyset\not= A\subseteq\R^n$ we define the {\em copolar set} by
$$ A^\star \coloneqq  \{x\in\R^n: \langle x,y\rangle\le -1 \mbox{ for all }y\in A\}$$
and the {\em shadow} of $A$ (imagining a light source at $o$) by
$${\rm shad}\,A\coloneqq  \{\lambda x:x\in A,\, \lambda\ge 1\}.$$

\begin{lemma}\label{L3.0}
Let $\emptyset\not=A\subseteq\R^n$. Then $A^\star\not=\emptyset$ if and only if $o\notin{\rm cl\,conv}\,A.$

Suppose that $o\notin{\rm cl\,conv}\,A$. Then $A^\star$ is a pseudo-cone, and $A^{\star\star} = {\rm shad\,cl\,conv}\,A$.
\end{lemma}

\begin{proof}
If $o\in A$, then $A^\star=\emptyset$. Otherwise,
$$ A^\star=\bigcap_{y\in A} H^-(y,-1),$$
from which it follows that $A^\star$ is either empty or a pseudo-cone. If $A^\star\not=\emptyset$, there is some $x\in\R^n\setminus\{o\}$ with $\langle x,y\rangle\le -1$ for all $y\in A$, thus $A\subseteq H^-(x,-1)$. This implies that $o\notin{\rm cl\,conv}\,A$. Conversely, if this holds, then $o$ and ${\rm cl\,conv}\,A$ can be strongly separated by a hyperplane, hence there is a vector $v$ such that $\langle v,y\rangle \le -1$ for all $y\in {\rm cl\,conv}\, A$. Then $v\in A^\star$ and thus $A^\star\not=\emptyset$.

Let $y\in A$. For all $x\in A^\star$ we have $\langle x,y\rangle\le -1$, hence $y\in A^{\star\star}$. Thus $A\subseteq A^{\star\star}$. Since $A^{\star\star}$ is a pseudo-cone, it follows that ${\rm shad\,cl\,conv}\,A \subseteq A^{\star\star}$.

Let $z\in\R^n$ be such that $z\notin {\rm shad\,cl\,conv}\,A$. Since the latter is a pseudo-cone, it does not intersect the closed segment $[o,z]$ with endpoints $o$ and $z$. Therefore, {\rm shad\,cl\,conv}\,A and $[o,z]$ can be strongly separated by a hyperplane (e.g., \cite[Thm. 1.3.7]{Sch14}), that is, there are a vector $v\not= o$ and a number $\tau<0$ such that $\langle v,x\rangle \le \tau$ for all $x\in {\rm shad\,cl\,conv}\,A$ and $\langle v,z\rangle >\tau$. After multiplying $v$ and $\tau$ by a suitable positive number, we may assume that $\tau=-1$. Then $\langle v,x\rangle\le -1$ for all $x\in A$ and hence $v\in A^\star$.  Since $\langle v,z\rangle >-1$, we deduce that $z\notin A^{\star\star}$. We have proved that $A^{\star\star} = {\rm shad\,cl\,conv}\,A$.
\end{proof}

Before turning to pseudo-cones, we show that also copolarity of closed convex sets has a linearization, similar as for the usual polarity of convex sets. For $A\subset\R^n$, we denote by ${\mathbbm 1}_A$ the characteristic function of $A$, that is,
$$ {\mathbbm 1}_A(x)\coloneqq  \begin{cases} 1, & \mbox{ if } x\in A,\\ 0, & \mbox{ if } x\in\R^n\setminus A.\end{cases}$$
Let $\C\C^n$ be the set of nonempty closed convex subsets of $\R^n$, let ${\sf U}(\C\C^n)$ be the set of finite unions of elements from $\C\C^n$ (including $\emptyset$), and let $V(\C\C^n)$ be the real vector space spanned by the characteristic functions of sets in $\C\C^n$. The following theorem and its proof are analogous to those for the ordinary polarity of convex sets (see, e.g., \cite[Sec. IV.1]{Bar02} or \cite[Thm. 1.8.2]{Sch22}), but for the reader's convenience we give the proof in full.

\begin{theorem}\label{T3.1}
There is a linear mapping $\phi_{\rm copol}: V(\C\C^n)\to V(\C\C^n)$ such that
$$ \phi_{\rm copol}({\mathbbm 1}_{K})={\mathbbm 1}_{K^\star} \quad\mbox{for } K\in {\C\C}^n.$$
\end{theorem}

\begin{proof}
There is a unique real valuation $\overline\chi$ on ${\sf U}(\C\C^n)$ satisfying $\overline \chi(K)=1$ for $K\in\C\C^n$ and $\overline\chi(\emptyset)=0$ (see, for example, \cite[Sec. 1.6]{Sch22}, in particular Thm. 1.6.8 and Note 2). Further (e.g., \cite[Thm. 1.6.2]{Sch22}), there is a linear mapping $\overline \chi:V(\C\C^n)\to \R$ with $\overline\chi({\mathbbm 1}_K)=1$ for $K\in\C\C^n$. For $y\in\R^n$ and $\varepsilon>0$ let $H_{y,\varepsilon}\coloneqq  \{x\in\R^n:\langle x,y\rangle\ge -1+\varepsilon\}$. For $g\in V(\C\C^n)$ define
$$ \phi_\varepsilon(g)(y)\coloneqq  \overline \chi(g) -\overline \chi(g{\mathbbm 1}_{H_{y,\varepsilon}})\quad\mbox{for }y\in\R^n.$$ 
The product $g{\mathbbm 1}_{H_{y,\varepsilon}}$ is an element of $V(\C\C^n)$, since $\C\C^n\cup \emptyset$ is closed under intersections. Thus, $\phi_\varepsilon$ is a linear mapping from $V(\C\C^n)$ into the vector space of real functions on $\R^n$. Now let $K\in {\C\C}^n$. Then we get
\begin{eqnarray*} 
\phi_\varepsilon({\mathbbm 1}_K)(y) &=& \overline\chi({\mathbbm 1}_K) - \overline\chi({\mathbbm 1}_{K\cap H_{y,\varepsilon}})\\
&=& \begin{cases}1, & \mbox{if } K\cap H_{y,\varepsilon}=\emptyset,\\ 0, & \mbox{if } K\cap H_{y,\varepsilon}\not=\emptyset,\end{cases}
= \begin{cases} 1, & \mbox{if } \langle x,y\rangle<-1+\varepsilon \;\forall\; x\in K,\\ 0, & \mbox{otherwise.}\end{cases}
\end{eqnarray*}
This implies that
\begin{eqnarray*} 
\lim_{\varepsilon\downarrow 0} \phi_\varepsilon({\mathbbm 1}_K)(y) &=&  \begin{cases} 1, & \mbox{if } \langle x,y\rangle \le -1\;\forall\; x\in K,\\ 0, & \mbox{otherwise, } \end{cases}\\
&=& {\mathbbm 1}_{K^\star}(y). 
\end{eqnarray*}
Hence, the limit $\phi_{\rm copol}(g)\coloneqq  \lim_{\varepsilon\downarrow 0}\phi_\varepsilon(g)$ exists for each $g\in V(\C\C^n)$ and defines a linear mapping $\phi_{\rm copol}:V(\C\C^n) \to V(\C\C^n)$ with $\phi_{\rm copol}({\mathbbm 1}_K) ={\mathbbm 1}_{K^\star}$. 
\end{proof}

Now we restrict ourselves to pseudo-cones. We remark that for a pseudo-cone $K$, Lemma \ref{L3.0} immediately gives $K^{\star\star}=K$. In the following lemma, the recession cone of $K$ need neither be pointed nor full-dimensional.

\begin{lemma}\label{L3a}
If $K\in ps\C^n$, then ${\rm rec}\,K^\star= ({\rm rec}\,K)^\circ$.
\end{lemma}

\begin{proof}
Let $K\in ps\C^n$. First, let $v\in ({\rm rec}\,K)^\circ$, hence $\langle v,w\rangle\le 0$ for all $w\in {\rm rec}\,K$. Let $x\in K^\star$ and $y\in K$. Then  $\langle x,y\rangle\le -1$ and $\langle v,y\rangle\le 0$ (since $y\in K\subset {\rm rec}\,K$), hence $\langle x+v,y\rangle \le -1$. Thus $x+v\in K^\star$. Since this holds for all $x\in K^\star$, we have $v\in {\rm rec}\,K^\star$. This proves that $({\rm rec}\,K)^\circ\subseteq {\rm rec}\, K^\star$.

Conversely, let $v\in{\rm rec}\,K^\star$. Then $x+v\in K^\star$ for all $x\in K^\star$, hence $\langle x+v,y\rangle\le -1$ for all $y\in K$. Since $\lambda y\in K$ for all $\lambda\ge 1$, we must have $\langle v,y\rangle\le 0$ for all $y\in K$. Let $w\in{\rm rec}\,K$. Then $\lambda w\in K$ for large $\lambda$, hence $\langle v,w\rangle\le 0$. It follows that $v\in ({\rm rec}\,K)^\circ$.
\end{proof}

Before continuing with copolarity of pseudo-cones, we give some references. Rashkovskii \cite[(3.1)]{Ras17} introduced copolarity for closed convex subsets with recession cone being (for the sake of simplicity, as he wrote) the nonnegative orthant, and used it to define and apply a certain copolar addition. Artstein--Avidan, Sadovsky and Wyczesany \cite {ASW22} developed a general theory of order reversing quasi involutions, as they called them, and gave many examples of involutions. Among them is (up to a reflection) the copolarity, there considered as the dual of the usual polarity of convex bodies containing the origin. (One may, however, notice that in a very special case and not under this name, copolarity already appeared in 1978, cf. Gigena \cite[Sec. 3]{Gig78}.) Xu, Li and Leng \cite{XLL22} have made a thorough study of copolarity and pseudo-cones. Their main result is the following. Let $n\ge 2$. A mapping $\tau$ from the set of pseudo-cones into itself satisfies $\tau(\tau(K))=K$ and $K\subset L\Rightarrow \tau(K)\supset\tau(L)$ for all pseudo-cones $K,L$ if and only if $\tau(K)=g(K^\star)$ for some $g\in {\rm GL}(n)$. They also observe (\ref{3.1})--(\ref{3.3}).

Let $K,L\in ps\C^n$. If $K\cap L\not=\emptyset$, then
\begin{eqnarray}\label{3.1}
(K\cap L)^\star= {\rm conv}(K^\star\cup L^\star).
\end{eqnarray}
If $o\notin {\rm conv}(K\cup L)$, then 
\begin{eqnarray}\label{3.2}
({\rm conv}(K\cup L))^\star= K^\star\cap L^\star.
\end{eqnarray}

Let $K$ be a $C$-pseudo-cone. The {\em radial function} $\rho(K,\cdot):{\rm int}\,C\to\R$ of $K$ is defined by
$$ \rho(K,x)\coloneqq  \min\{\lambda\in\R:\lambda x\in K\}\quad\mbox{for }x\in{\rm int}\,C.$$
Then  $0< \rho(K,x)<\infty$ for $x\in{\rm int}\,C$. (Note that $x\in{\rm int}\,C$ implies $\lambda x \in K$ for some $\lambda\ge 0$, since $C$ is the recession cone of $K$.) We have
\begin{equation}\label{3.3} 
\rho(K,x)= \frac{-1}{h(K^\star,x)} \quad\mbox{for }x\in {\rm int}\,C.
\end{equation}

First we supplement now these observations by some remarks. For $u\in S^{n-1}$ and $t<0$, the closed halfspace $H^-(u,t)$ is a pseudo-cone. We have $H^-(u,t)^\star =\R_{\ge 1/|t|}u$, where we use the abbreviation $\R_{\ge \lambda}\coloneqq \{r\in\R:r\ge \lambda\}$. If $K\in ps\C^n$ is a pseudo-cone with $K\cap H^-(u,t)\not=\emptyset$, it follows from (\ref{3.1}) that
$$ (K\cap H^-(u,t))^\star = {\rm conv}(K^\star \cup \R_{\ge 1/|t|}u).$$
On the other hand, $H^-(u,0)$ is not a pseudo-cone, but $K\cap H^-(u,0)$ is still a pseudo-cone, if it is not empty. Its copolar set is given by the following lemma.

\begin{lemma}\label{L3.1}
Let $K$ be a $C$-pseudo-cone and $H^-(u,0)$ a closed halfspace such that $K\cap H^-(u,0)\not=\emptyset$. Then 
$$ (K\cap H^-(u,0))^\star = K^\star+\R_{\ge 0}u.$$
\end{lemma}

\begin{proof}
Let $y\in K^\star+\R_{\ge 0}u$ and $ y'=\lambda y$ with $\lambda\ge 1$. Then $y=x+\mu u$ with $x\in K^\star$ and $\mu\ge 0$. It follows that $y'=\lambda x+\lambda\mu u\in K^\star+\R_{\ge 0}u$. 

Let $y\in K^\star+\R_{\ge 0}u$ be as above, that is, $y=x+\mu u$ with $x\in K^\star$ and $\mu\ge 0$.  For $z\in K\cap H^-(u,0)$ we then have $\langle x,z\rangle\le -1$ and $\langle u,z\rangle\le 0$, hence $\langle y,z\rangle \le -1$. This means that $y\in (K\cap H^-(u,0))^\star$.

Conversely, suppose that $y\in\R^n$ and $y\notin K^\star+\R_{\ge 0}u$. The segment $[y,o]$ with endpoints $y$ and $o$ does not meet $K^\star+\R_{\ge 0}u$, by the property shown first. Therefore, $K^\star+\R_{\ge 0}u$ and $[y,o]$ can be strongly separated by a hyperplane. As in the proof of Lemma \ref{L3.0}, there is a vector $v\not=o$ such that $K^\star+\R_{\ge 0}u\subset{\rm int}\,H^-(v,-1)$ and $\langle y,v\rangle >-1$. Since $\langle x,v\rangle \le -1$ for all $x\in K^\star$, we have $v\in K^{\star\star}=K$. Let $x\in K^\star$. Since $\langle x+\lambda u,v\rangle \le -1$ for all $\lambda>0$, we have $\langle u,v\rangle\le 0$. 
This shows that $v\in K \cap H^-(u,0)$. Since $\langle y,v\rangle >-1$, this implies that $y\notin (K\cap H^-(u,0))^\star$, which completes the proof.
\end{proof}

Next, we remark that relation (\ref{3.3}) can be slightly strengthened in a useful way. For this, we define:

\begin{definition}
Let $K$ be a pseudo-cone. A pair $(x,v)\in\R^n\times\R^n$ is a {\em crucial pair} of $K$ if $x\in\partial K$ and $v$ is an outer normal vector of $K$ at $x$ such that  $\langle x,v\rangle =-1$.
\end{definition}

\begin{lemma}\label{L3.2}
If $(x,v)$ is a crucial pair of the pseudo-cone $K$, then $(v,x)$ is a crucial pair of $K^\star$.
\end{lemma}

\begin{proof}
Let $(x,v)$ be a crucial pair of $K\in ps\C^n$. Since $v$ is an outer normal vector of $K$ at $x\in\partial K$, for all $y\in K$ we have $\langle y-x,v\rangle\le 0$ and hence $\langle y,v\rangle\le -1$. Therefore, $v\in K^\star$. Since $x\in K= K^{\star\star}$, we have $\langle z,x\rangle\le-1$ for all $z\in K^\star$, thus $\langle z-v,x\rangle\le 0$ for $z\in K^\star$. From $v\in K^\star$ it now follows that $v\in\partial K^\star$ and that $x$ is an outer normal vector of $K^\star$ at $v$. Thus $(v,x)$ is a crucial pair of $K^\star$. 
\end{proof}

Relation (\ref{3.3}) for $x\in{\rm int}\,C$ follows from Lemma \ref{L3.2}. Since the support function is homogeneous of degree $1$ and the radial function is homogeneous of degree $-1$, it suffices to consider an argument $x\in\partial K$. Let $v$ be such that $(x,v)$ is a crucial pair of $K$ and hence $(v,x)$ is a crucial pair of $K^\star$. Then $\rho(K,x)=1$ and $h(K^\star,x)=\langle v,x\rangle =-1$. 

Also useful is the following reformulation. Let $K$ be a $C$-pseudo-cone. For $x\in{\rm int}\,C$ we have $x\in\partial K \Leftrightarrow \rho(K,x)=1\Leftrightarrow h(K^\star,x)=-1$, hence
\begin{equation}\label{3.4}
x\in\partial K\cap {\rm int}\,C \Leftrightarrow H(x,-1) \mbox{ is a supporting hyperplane of }K^\star.
\end{equation}

An {\em exposed face} of a nonempty closed convex set is, by definition, the intersection of the set with one of its supporting hyperplanes. Exposed faces of convex bodies behave well under the ordinary polarity. The same holds true for copolarity of pseudo-cones, but here we must distinguish between different types of exposed faces. 

Let $K$ be a $C$-pseudo-cone. For an exposed face $F$ of $K$ we define the {\em conjugate face} of $F$ by
$$ \widehat F\coloneqq  \{x\in K^\star: \langle x,y\rangle=-1 \mbox{ for all }y\in F\}.$$
(One should keep in mind that $\widehat F$ depends not only on $F$ but also on $K$, which is not expressed by the notation.)

We note that, by definition, any exposed face $F$ of $K$ can be written in the form $F=K\cap H(u,t)$ with $K\subset H^-(u,t)$ and hence $u\in C^\circ$ and $t\le 0$. Here $t=0$ is only possible if $u\in\partial C^\circ$ (and hence $F\subset \partial C$), since otherwise $C$ could not be the recession cone of $K$. If $t<0$, we can normalize the vector $u$ so that $t=-1$. If $F=K\cap H(u,-1)$, then automatically $K\in H^-(u,-1)$, since it follows from the pseudo-cone property that any supporting hyperplane of $K$ not passing through $o$ strictly separates ${\rm int}\,K$ and $o$.

Therefore, we can also write
$$ \widehat F=\{u\in\R^n\setminus\{o\}: F=K\cap H(u,-1)\}.$$
In fact, for $u\in\R^n\setminus\{o\}$ we have
\begin{eqnarray*}
u\in\widehat F &\Leftrightarrow& u\in K^\star,\,\langle u,y\rangle =-1 \mbox{ for all }y\in F\\
&\Leftrightarrow& \langle u,z\rangle \le -1 \mbox{ for }z\in K,\, \langle u,y\rangle=-1 \mbox{ for }y\in F\\
&\Leftrightarrow&K\subset H^-(u,-1),\, F\subset H(u,-1)\\
&\Leftrightarrow& F=K\cap H(u,-1),
\end{eqnarray*}
from which the assertion follows.

If $F$ is unbounded and contained in $\partial C$, then $\widehat F=\emptyset$. Therefore, we have to distinguish between different types of exposed faces, and we consider the following sets. By $\F_b^{\,\rm in}(K)$ we denote the set of bounded exposed faces meeting the interior of $C$, and by $\F_b^\partial(K)$ the set of bounded exposed faces contained in the boundary of $C$. The set of unbounded exposed faces of $K$ meeting ${\rm int}\,C$ is denoted by $\F_u^{\,\rm in}(K)$. 

\begin{theorem}\label{T3.2} Let $K$ be a $C$-pseudo-cone, and let $F$ be an exposed face of $K$, which is either bounded or meets the interior of $C$. Then $\widehat F$ is an exposed face of $K^\star$, more precisely,
$$ F\in \F_b^{\,\rm in}(K)\Rightarrow \widehat F\in\F_b^{\,\rm in}(K^\star),$$
$$ F\in \F_b^\partial(K)\Rightarrow \widehat F\in \F_u^{\,\rm in}(K^\star),\qquad
F\in\F_u^{\,\rm in}(K)\Rightarrow \widehat F\in\F_b^\partial(K^\star).$$
Further, $\widehat{\widehat F} =F$ and $\dim F +\dim\widehat F=n-1$.
\end{theorem}

\begin{proof}
Let $F$ be an exposed face of $K$ that meets ${\rm int}\,C$, so that $F=K\cap H(u,-1)$ for a suitable $u$. Then $u\in \widehat F$ and hence $\widehat F\not=\emptyset$. Choose $y\in F\cap{\rm int}\,C$. Then $\rho(K,y)=1$, hence $h(K^\star,y)=-1$ by (\ref{3.3}). Thus, 
 the hyperplane $H(y,-1)$ supports $K^\star$. Hence, $\widehat{\{y\}} = K^\star\cap H(y,-1)$ is an exposed face of $K^\star$. It follows that
$$\widehat F=\bigcap_{y\in F} (K^\star\cap H(y,-1)) = \bigcap_{y\in F\cap{\rm int}\,C} (K^\star\cap H(y,-1)) = \bigcap_{y\in F\cap{\rm int}\,C} \widehat{\{y\}}.$$
Here we have used that for $\emptyset\not= M\subset\R^n$ we have $\bigcap_{y\in M}H(y,-1)=\bigcap_{y\in{\rm cl}\,M}H(y,-1)$. In fact, if $\langle x,y\rangle\not=-1$ for some $y\in{\cl}\,M$, then also $\langle x,y'\rangle\not=-1$ for some $y'\in M$. Thus, $\widehat F$ is an intersection of exposed faces and hence an exposed face of $K^\star$, by \cite[Thm. 2.1.3]{Sch14}. Assume first, in addition, that $F$ is bounded. Suppose that $\widehat F\subset\partial C^\circ$. Then there are unit vectors $u\in{\rm int}\,C$ for which $h(K^\star,u)$ comes arbitrarily close to zero. By (\ref{3.3}), the radial function of $F$ on unit vectors $u$ with $\rho(K,u)u\in F$ would then be unbounded, a contradiction. Hence, $\widehat F$ meets the interior of $C^\circ$. Now assume, second, that $F$ is unbounded. As remarked above, we can write $F=K\cap H(u,-1)$, and then $u\in\partial C^\circ$, since otherwise $C\cap H(u,-1)$ would be bounded. It follows that $\widehat F\subset\partial C^\circ$.

Finally, let $F=K\cap H(u,-1)$ be a bounded exposed face of $K$ that is contained in $\partial C$. We cannot have $u\in\partial C^\circ$, since that would imply that $F$ is either empty or unbounded. Thus, $u\in {\rm int}\,C^\circ$. Since $u\in \widehat F$, we see that $\widehat F$ is not empty and meets ${\rm int}\,C^\circ$. Let $p\in{\rm relint}\,F$ (the relative interior of $F$). Since $p\in\partial C$, there is a supporting hyperplane $H$ of $C$ with $p\in H$. Since $p\in{\rm relint}\,F$,  it follows that $F\subset H$. The intersection $G\coloneqq  C\cap H$ is an exposed face of $C$, and we have $F\subset G$. Let
$$ G^\lhd\coloneqq  \{y\in  C^\circ: \langle y,x\rangle =0\mbox{ for all } x\in G\}.$$
Choose $y_0\in\widehat F$. We state that
\begin{equation}\label{3.5}
G^\lhd + y_0\subset\widehat F.
\end{equation}
For the proof, let $z+y_0\in G^\lhd+y_0$. Then $z\in C^\circ$ (hence $\langle z,x\rangle\le 0$ for all $x\in C$) and $\langle z,x\rangle=0$ for all $x\in G$. Hence, for all $x\in K$ we obtain $\langle z+y_0,x\rangle\le -1$ and thus $z+y_0\in K^\star$, and for all $x\in F$ we obtain $\langle z+y_0,x\rangle =-1$ and thus $z+y_0\in\widehat F$. We have proved (\ref{3.5}), which implies that $\widehat F$ is unbounded.

To prove that $\widehat{\widehat F} =F$, let $y\in F$. Then $\langle x,y\rangle =-1$ for all $x\in\widehat F$. Since $y\in K= K^{\star\star}$, it follows that $y\in \widehat{\widehat F}$. Thus $F\subseteq \widehat{\widehat F}$. Each exposed face $F$ of $K$, except those which are unbounded and contained in $\partial C$, can be written as $F=K\cap H(u,-1)$, where $u\in\widehat F$. Let $z\in\widehat{\widehat F}$. Then $z\in K^{\star\star}=K$ and $\langle z,x\rangle =-1$ for all $x\in\widehat F$, in particular for $u$. Thus $\langle z,u\rangle =-1$, hence $z\in F$. This shows that $\widehat{\widehat F}\subseteq F$.

Suppose that $F$ is of dimension $k\in\{0,\dots,n-1\}$. Then there are $k+1$ affinely independent points $y_1,\dots,y_{k+1} \in F\cap C$. By Lemma \ref{L3.2}, the linearly independent vectors $y_1,\dots,y_{k+1}$ are normal vectors of supporting hyperplanes of $P^\star$ containing $\widehat F$. It follows that $\dim\widehat F\le n-1-k$. Conversely, $P$ has $n-k$ linearly independent normal vectors at $F$, hence $\widehat F$ contains $n-k$ affinely independent points, which shows that $\dim \widehat F\ge n-1-k$. We have proved that $\dim \widehat F+\dim F= n-1$.
\end{proof}

\section{Polyhedral pseudo-cones}\label{sec4}

A pseudo-cone is {\em polyhedral} if it is the intersection of finitely many closed halfspaces. First we want to describe the copolar set of a polyhedral pseudo-cone $P\in ps\C^n$. We have 
$$ P = \bigcap_{i=1}^k H^-(u_i,t_i) \cap \bigcap_{j=1}^m H^-(v_j,0),$$
with unit vectors $u_i,v_j$, numbers $t_i<0$ and integers $k\ge 1$ and $m\ge 0$. It follows from (\ref{3.1}) and Lemma \ref{L3.1} that
$$ P^\star ={\rm conv}\left(\bigcup_{i=1}^k \R_{\ge 1/|t_i|}u_i\right)+ \sum_{j=^1}^m \R_{\ge 0}v_j.$$
In particular, $P^\star$ is also polyhedral.

The proper faces of a polyhedral pseudo-cone are all exposed faces. For polyhedral pseudo-cones, we restate Theorem \ref{T3.2} in the following way. That the mappings are inclusion-reversing, follows from the definition of $\widehat F$. The involution property means that $\widehat{\widehat F}=F$. 
\begin{theorem}\label{T4.1}
Let $P\in ps\C^n$ be a polyhedral pseudo-cone whose recession cone is pointed and $n$-dimensional. \\[1mm]
The mapping $F\mapsto\widehat F$, restricted to $\F_b^{\,\rm in}(P)$,  is an inclusion-reversing involution onto $\F_b^{\,\rm in}(P^\star)$. \\[1mm]
The mapping $F\mapsto\widehat F$, restricted to $\F_b^\partial(P)$, is an inclusion-reversing involution onto $\F_u^{\,\rm in}(P^\star)$.\\[1mm]
The mapping $F\mapsto\widehat F$, restricted to  $\F_u^{\,\rm in}(P)$, is an inclusion-reversing involution onto $\F_b^\partial(P^\star)$.
\end{theorem}
We remark that if $P$ is a polyhedral $C$-pseudo-cone (so that also $C$ is polyhedral), then the totality of faces of $P^\star$ can be obtained from \cite[Lem. 1.5.4]{Sch22}. Since $C^\circ$ is the recession cone of $P^\star$, it implies that
$$ P^\star= \left(\bigcup_{F\in \F_b(P^\star)} F\right) + C^\circ,$$
where $\F_b(P)$ denotes the set of bounded faces of a polyhedral set $P$. Hence, each face of $P^\star$ is the sum of suitable faces $F\in\F_b(P^\star)$ and $G\in\F(C^\circ)$. The elements of $\F(C^\circ)$ are the normal cones of the faces of $C$.

The local geometry of polyhedral sets, beyond faces, is determined by normal cones and angle cones. They can be described as follows. Suppose that $P=\bigcap_{i=1}^m H_i^-$ with closed halfspaces $H_1^-,\dots, H_m^-$ and that $F$ is a face of $P$. Without loss of generality, let $H_1^-,\dots,H_k^-$ be the halfspaces that contain $F$ in their boundaries, and let $u_1,\dots,u_k$ be the outer unit normal vectors of these halfspaces. Then the normal cone of $P$ at $F$ is given by
$$ N(P,F)= {\rm pos}\{u_1,\dots,u_k\},$$
where ${\rm pos}$ denotes the positive hull. Further, with an arbitrary $z\in{\rm relint}\,F$, the angle cone of $P$ at $F$ can be defined by 
$$ A(F,P)\coloneqq  \bigcap_{i=1}^k(H_i^--z) = {\rm pos}(P-z).$$
We have $N(P,F)= A(F,P)^\circ$. (For more on these cones, see \cite[Sec. 1.4]{Sch22}.)

If now $P$ is a polyhedral $C$-pseudo-cone and $F\in \F_b(P)\cup \F_u^{\,\rm in}(P)$ is a face of $P$, we have
$$ N(P^\star,\widehat F)= {\rm pos}\,F.$$
This can be deduced from (\ref{3.4}), which remains true if the left-hand side is replaced by $x\in{\rm cl}(\partial K\cap {\rm int}\,C)$. It follows that also
$$ A(\widehat F, P^\star) = ({\rm pos}\,F)^\circ.$$

With the notation used above, the set
$$ T(F,P)\coloneqq  {\rm pos}(P-z)+z =A(F,P) +z$$
with $z\in{\rm relint}\,F$ is known as the tangent cone of $P$ at $F$ (although, strictly speaking, it is the translate of a cone). Thus, $T(F,P)$ is the intersection of all supporting halfspaces of $P$ that contain the face $F$ in their boundary. We have
$$ T(\widehat F,P^\star)= ({\rm shad}\,F)^\star.$$
This can also be deduced from (\ref{3.4}).

\section{Estimates for $C$-close sets}\label{sec5}

A $C$-close set is, by definition, a $C$-pseudo-cone $K$ with $V_n(C\setminus K)<\infty$. It is to be expected that the support function of such a set approaches zero in a controllable way when the boundary of $\Omega_{C^\circ}$ is approached. The following theorem makes this precise. Since we may apply a dilatation, it suffices to consider $C$-close sets $K$ with $V_n(C\setminus K)=1$. Recall that $\delta_C(u)$ denotes the spherical distance of $u\in\Omega_{C^\circ}$ from the boundary of $\Omega_{C^\circ}$.

\begin{theorem}\label{T5.1}
Choose a number $0<\alpha_0<\pi/2$. There is a constant $c_1$, depending only on $C$ and $\alpha_0$, such that every $C$-close set $K$ with $V_n(C\setminus K)=1$ satisfies
$$ \overline h_K(u) \le c_1\delta_C(u)^{1/n}\quad\mbox{for }u\in\Omega_{C^\circ}\mbox{ with }\delta_C(u)\le \alpha_0.$$
\end{theorem}

\begin{proof}
 We start with a pair $v\in\partial C^\circ$ and $w\in\partial C$ of orthogonal unit vectors (note that to any unit vector $v\in\partial C^\circ$ there is an orthogonal
unit vector $w\in\partial C$, as follows from the Moreau decomposition; see, e.g., \cite[Thm. 1.3.3]{Sch22}). The vectors $v$ and $w$ span a two-dimensional linear subspace, which we denote by $E$. For each $\alpha\in (0,\alpha_0]$ there is a unique unit vector $u_\alpha$ such that $v\in{\rm pos}\,\{w,u_\alpha\}$ and $\angle(v,u_\alpha)=\alpha$. The hyperplane $H(u_\alpha,-1)$ intersects the ray $\R_{\ge 0}w$ in a point $x_\alpha$. The hyperplane $H(u_\alpha,-1)$ touches the unit sphere $S^{n-1}$ in a point $z_\alpha$, and the ray $w-\R_{\ge 0}v$ intersects the hyperplane $H(u_\alpha,-1)$ in a point $y_\alpha$. We define
$$ a_1=a_1(\alpha)\coloneqq  \|w-y_\alpha\|,\quad a_2=a_2(\alpha)\coloneqq  \|w-x_\alpha\|.$$
The points $o,v,w,x_\alpha,y_\alpha,z_\alpha$ all lie in $E$, and the right-angled triangle with vertices $w,x_\alpha,y_\alpha$ has at $x_\alpha$ the angle $\alpha$, hence
$$ a_1=a_2\tan\alpha,\quad 1= (a_2+1)\sin\alpha,\quad a_1= \frac{1-\sin\alpha}{\cos\alpha}.$$
For $\alpha\in(0,\alpha_0]$ it follows that
$$ a_1(\alpha)\ge a_1(\alpha_0)\eqqcolon c_2.$$

In the following, we write
$$ H^+(u_\alpha,-1)\coloneqq \{x\in\R^n: \langle x,u_\alpha\rangle\ge -1\}$$
and consider the $(n-1)$-dimensional convex set
$$ A(v,w,\alpha)\coloneqq  C\cap H(w,1)\cap H^+(u_\alpha,-1).$$
Regarding the positive function $(v,w)\mapsto V_{n-1}(A(v,w,\alpha_0))$, defined on orthogonal pairs of unit vectors $v\in\partial C^\circ$, $w\in\partial C$, we see from continuity considerations that it cannot come arbitrarily close to $0$, hence there exists a constant $c_3>0$, depending only on $C$ and $\alpha_0$, such that $V_{n-1}(A(v,w,\alpha_0))\ge c_3$.
For $\alpha\in(0,\alpha_0]$ we have $A(v,w,\alpha)\supseteq A(v,w,\alpha_0)$, hence 
$$ V_{n-1}(A(v,w,\alpha))\ge c_3.$$
The set $C\cap H^+(u_\alpha,-1)$ contains the convex hull of $x_\alpha$ and $A(v,w,\alpha)$, hence
\begin{equation}\label{5.1} 
V_n(C\cap H^+(u_\alpha,-1)) \ge \frac{1}{n} a_2V_{n-1}(A(v,w,\alpha_0)) \ge \frac{c_3}{n} \frac{1-\sin\alpha}{\sin\alpha} \ge \frac{c_4}{\alpha},
\end{equation}
with a constant $c_4$ depending only on $C$ (which implies the dependence on $n$) and $\alpha_0$.

Now we start with an arbitrary point $u_\alpha\in \Omega_{C^\circ}$ with $\delta_C(u_\alpha)=\alpha\le \alpha_0$. Let $v\in\partial \Omega_{C^\circ}$ be a point with smallest spherical distance from $u_\alpha$, so that $\angle(u_\alpha,v)=\alpha$. Let $w$ be the unit tangent vector of the circular arc connecting $u_\alpha$ and $v$, oriented so that it points away from $u_\alpha$. Then $w$ is an outer normal vector of a supporting hyperplane to $C^\circ$ at $v$, since otherwise there would be points in $\partial \Omega_{C^\circ}$ closer (in spherical distance) to $u_\alpha$ than $v$. It follows that $w\in\partial C$. 

The supporting hyperplane $H(K,u_\alpha)$ of $K$ with outer normal vector $u_\alpha$ has distance $\overline h_K(u_\alpha)$ from the origin. Further, denoting by $H^+(K,u_\alpha)$ the closed halfspace bounded by $H(K,u_\alpha)$ and not containing $K$, we have
$$V_n(C\cap H^+(K,u_\alpha))\le V_n(C\setminus K)=1.$$
Applying, to a suitable situation considered above, the dilatation with factor $\overline h_K(u_\alpha)$, we see that
$$ \overline h_K(u_\alpha)^n V_n(C\cap H^+(u_\alpha,-1))\le 1.$$
Together with (\ref{5.1}), this yields $\overline h_K(u_\alpha)^nc_4/\alpha\le 1$, and since $\alpha=\delta_C(u_\alpha)$, this proves the assertion.
\end{proof}

From this theorem and the subsequent lemma, we can draw a conclusion about the convergence of $C$-close sets. For $\tau>0$, we write 
\begin{equation}\label{5.0}
\overline\omega(\tau)\coloneqq \{u\in\Omega_{C^\circ}: \delta_C(u)\ge \tau\}.
\end{equation}

\begin{lemma}\label{L5.1}
If $K$ is a $C$-pseudo-cone and $\tau>0$, then the reverse spherical image ${\bm x}_K(\overline\omega(\tau))$ satisfies
$$ {\bm x}_K(\overline\omega(\tau)) \subset \frac{b(K)}{\sin\tau}B^n.$$
\end{lemma}

\begin{proof}
Let $K$ be a $C$-pseudo-cone and $x\in{\bm x}_K(\overline\omega(\tau))$, let $u\in\overline\omega(\tau)$ be an outer normal vector of $K$ at $x$. Let $x'$ be such that $x=\|x\|x'$.
Then
$$ \|x\| |\langle x',u\rangle| =|\langle x,u\rangle|=\overline h (K,u)\le b(K),$$
by (\ref{2.1}). If $\gamma$ denotes the angle between $x'$ and $u$, we have $\gamma\ge (\pi/2)+\alpha+\beta$ with $\alpha\ge \tau$ and $\beta\ge 0$. This gives
$$ \langle x',u\rangle = \cos\gamma \le -\sin(\alpha+\beta) \le -\sin \tau.$$
Thus, we get
$$ \|x\|\le \frac{b(K)}{|\langle x',u\rangle|} \le \frac{b(K)}{\sin \tau}$$
and, therefore, the assertion.
\end{proof}

In the following lemma, a well-known property of the convergence of convex bodies is carried over to $C$-close pseudo-cones.

\begin{lemma}\label{L5.2}
Suppose that $(K_i)_{i\in\N}$ is a sequence of $C$-close sets with $V_n(C\setminus K) \le 1$, converging to a $C$-close set $K_0$. Then the sequence $(\overline h_{K_i})_{i\in\N}$ converges uniformly to $\overline h_{K_0}$.
\end{lemma}

\begin{proof}
Let $\alpha_0$ and $c_1$ be as in Theorem \ref{T5.1}. Let $\varepsilon>0$ be given. We choose a number $0<\tau<\alpha_0$ with $c_1\tau^{1/n}<\varepsilon$. Since $K_i\to K_0$, there is a number $b$ with $b(K_i)< b$ for $i\in\N_0$, hence Lemma \ref{L5.1} yields
$$ {\bm x}_{K_i}(\overline\omega(\tau)) \subset \frac{b}{\sin\tau}B^n.$$
We choose a number $t$ with $\frac{b}{\sin\tau}B^n\cap C\subset C^-(t)$. Then we have
$$ {\bm x}_{K_i}(\overline\omega(\tau)) \subset C^-(t) \quad\mbox{for all }i\in\N_0.$$
Since
$$ K_i\cap C^-(t)\to K_0\cap C^-(t)\quad\mbox{as }i\to\infty$$
is a convergence of ordinary convex bodies, we have $h_{K_i\cap C^-(t)}\to h_{K_0\cap C^-(t)}$ uniformly on $S^{n-1}$ (see, e.g., \cite[Sec. 1.8]{Sch14}). In particular, since
$$ h(K_i\cap C^-(t),u) = h(K_i,u) \quad\mbox{for }u\in\overline\omega(\tau),\ i\in\N_0,$$
this means that there exists a number $i_0\in\N$ such that
$$ |\overline h_{K_i}(u)- \overline h_{K_0}(u)|<\varepsilon\quad\mbox{for }  i\ge i_0,\,u\in\overline \omega(\tau).$$

For $u\in\Omega_{C^\circ}\setminus\overline \omega(\tau)$ we have $\delta_C(u)<\tau<\alpha_0$ and hence, by Theorem \ref{T5.1}, $\overline h_{K_i}(u)\le c_1\delta_C(u)^{1/n}< c_1\tau^{1/n}< \varepsilon$ for $i\in\N_0$ and thus (since $\overline h_{K_i}(u), \overline h_{K_0}(u)>0$)
$$ |\overline h_{K_i}(u)-\overline h_{K_0}(u)|<\varepsilon \quad \mbox{for all }i\in\N.$$
This completes the proof.
\end{proof}

The following lemma will be needed in the next section.

\begin{lemma}\label{L5.3}
Let $K\in ps\C^n$, and let $\tau>0$ be such that $S_{n-1}(K,\overline \omega(\tau)) \eqqcolon s>0$. There is a constant $b_0>0$, depending only on $C,\tau$ and $s$, such that $b(K)\ge b_0$.
\end{lemma}

\begin{proof}
Let $K\in ps\C^n$, and assume first that $b(K)=1$. There is a point $y\in K\cap\partial B^n$. Let $x\in {\bm x}_K(\overline\omega(\tau))$. Then there is a supporting hyperplane $H(K,u)$ of $K$ at $x$ with $u\in\overline\omega(\tau)$. It separates $o$ and $y$. Therefore, and since $\delta_C(u)\ge \tau>0$, there is a constant $t$, depending only on $C$ and $\tau$, such that $C\cap H(K,u)\subset C^-(t)$. In particular, ${\bm x}_K(\overline\omega(\tau))\subset C^-(t)$. This implies that
$ S_{n-1}(K,\overline\omega(\tau))< S(C^-(t))$, where $S$ denotes the surface area. 

Now let $K\in ps\C^n$ be arbitrary. Then $b(b(K)^{-1}K)=1$ and hence 
$$ S(C^-(t)) > S_{n-1}(b(K)^{-1}K,\overline\omega(\tau)) = b(K)^{-(n-1)}S_{n-1}(K,\overline\omega(\tau)) =  b(K)^{-(n-1)} s,$$
thus $b(K)^{n-1}> s/S(C^-(t))\eqqcolon b_0^{n-1}$.
\end{proof}

\section{Pseudo-cones with given surface area measures}\label{sec6}

Let $\varphi$ be a nonzero, locally finite Borel measure on $\Omega_{C^\circ}$. If it is allowed to be infinite, then it need not be the surface area measure of some $C$-pseudo-cone, if no extra conditions are imposed. A necessary growth condition was found in \cite{Sch21} (it does not need the `asymptotic' assumption made in \cite{Sch21}). We mention that a moderate growth condition, but for functions, also appears in Pogorelov \cite{Pog80}. This author, and more generally Chou and Wang \cite{CW95}, were interested, from the PDE viewpoint, in unbounded complete convex $C_+^2$ hypersurfaces with given Gauss curvature on the spherical image.

In the following, we deal with the surface area measures of $C$-close sets. If $K$ is such a set, it is to be expected that the finiteness of the volume of $C\setminus K$ imposes stronger restrictions  on the surface area measure of $K$.

We have already mentioned after Theorem \ref{T1.1} that a $C$-close set with given surface area measure is uniquely determined. We emphasize that $C$-pseudo-cones with given surface area measure (even if it is finite) need not be unique if they are not $C$-close. For example, if $\varphi$ is concentrated in a one-pointed set $\{u\}$, we can choose an arbitrary $(n-1)$-dimensional convex body $F$ with $(n-1)$-dimensional volume equal to $\varphi(\{u\})$ and a rigid motion $g$ such that $gF\subset C$ and $gF$ is orthogonal to $u$. Then $gF+C$ is a $C$-pseudo-cone with surface area measure $\varphi$. A $C$-full set with this surface area measure is obtained if we choose $gF=C\cap H$ with a hyperplane $H$ orthogonal to $u$ and such that $C\cap H$ has $(n-1)$-dimensional volume equal to $\varphi(\{u\})$.

The proof of the sufficiency of the condition 
\begin{equation}\label{6.0} 
\int_{\Omega_{C^\circ}} \delta_C^{1/n}\,\D\varphi<\infty,
\end{equation}
is based on the existence theorem for finite measures with compact support, proved in \cite{Sch18}. Therefore, we start with the following definitions for a compact set $\overline\omega\subset\Omega_{C^\circ}$. We say that a nonempty convex set $K$ is {\em $C$-determined by $\overline\omega$} if
$$ K = C\cap \bigcap_{u\in\overline\omega} H^-(K,u),$$
where $H^-(K,u)$ is the supporting halfspace of $K$ with outer unit normal vector $u$. By $\K(C,\overline\omega)$ we denote the family of all sets that are $C$-determined by $\overline\omega$. These sets are special pseudo-cones, namely $C$-full sets. 

The following proposition is the essence of the proof of Theorem 3 in \cite{Sch18}.

\begin{proposition}\label{P6.1} 
Let $\overline\omega\subset\Omega_{C^\circ}$ be compact, and let $\varphi$ be a nonzero finite Borel measure on $\Omega_{C^\circ}$ with support contained in $\overline\omega$. There is a set $M\in\K(C,\overline \omega)$ satisfying
$$ V_n(C\setminus M) =1$$
and such that the set
$$ K\coloneqq  \lambda^{\frac{1}{n-1}}M \quad\mbox{with}\quad \lambda\coloneqq \frac{1}{n}\int_{\overline\omega} \overline h_M\,\D\varphi$$
satisfies $\varphi= S_{n-1}(K,\cdot)$.
\end{proposition}

\vspace{1mm}

\noindent{\em Proof of Theorem} \ref{T1.1}. We assume that $\varphi$ is a nonzero Borel measure on $\Omega_{C^\circ}$ satisfying (\ref{6.0}). To apply Proposition \ref{P6.1}, we choose a sequence $(\omega_j)_{j\in\N}$ of open sets $\omega_j\subset\Omega_{C^\circ}$ with $\overline \omega_j\coloneqq {\rm cl}\,\omega_j\subset \omega_{j+1}$ for $j\in\N$ and $\bigcup _{j\in\N}\omega_j=\Omega_{C^\circ}$. Then we define, for each $j\in\N$, a measure $\varphi_j$ by $\varphi_j(\omega)\coloneqq  \varphi(\omega\cap \omega_j)$ for $\omega\in\B(\Omega_{C^\circ})$. Since (\ref{6.0}) holds and $\delta_C$ is bounded away from zero on $\omega_j$, the measure $\varphi_j$ is finite, and its support is contained in $\overline\omega_j$. By an appropriate choice of $\omega_1$ we can also achieve that $\omega_1\supset \overline\omega(\tau)$ for some $\tau>0$ with $\varphi(\overline\omega(\tau))>0$. Then, in particular, each $\varphi_j$, is not the zero measure.

For each $j\in\N$, Proposition \ref{P6.1} now yields the existence of a convex set $M_j\in\K(C,\overline \omega_j)$ satisfying
\begin{equation}\label{6.2}
V_n(C\setminus M_j) =1
\end{equation}
and such that the set
\begin{equation}\label{6.3} 
K_j\coloneqq  \lambda_j^{\frac{1}{n-1}}M_j \quad\mbox{with}\quad \lambda_j\coloneqq \frac{1}{n}\int_{\overline\omega_j} \overline h_{M_j}\,\D\varphi
\end{equation}
satisfies $\varphi_j= S_{n-1}(K_j,\cdot)$.

The sets $M_j$ do not escape to infinity. This follows from the fact that $V_n(C\setminus M_j) =1$, since it implies that the sequence $(M_j)_{j\in\N}$ has bounded distances from the origin. It also follows from the choice of $\omega_1$ and from Lemma \ref{L5.3} that for the sequence $(M_j)_{j\in\N}$ the distances from the origin are bounded away from $0$. Hence the sequence has, by Lemma \ref{L2.2}, a subseqence that converges to a $C$-pseudo-cone $M$. After renumbering, we assume that the sequence $(M_j)_{j\in\N}$ itself converges to $M$. It follows from Lemma \ref{L2.3} that $V_n(C\setminus M)\le 1$. 

Now we state that there is a constant $c_5$, independent of $j$, such that
\begin{equation}\label{6.1}
\int_{\Omega_{C^\circ}} \overline h_{M_j}\,\D\varphi< c_5.
\end{equation}
For the proof, we use $\overline\omega(\alpha_0)$ as defined by (\ref{5.0}), with the constant $\alpha_0$ appearing in Theorem \ref{T5.1}. By Lemma \ref{L5.2} we have $\overline h_{M_j}\to \overline h_{M}$ uniformly on the compact set $\overline \omega(\alpha_0)$, hence
$$ \int_{\overline \omega(\alpha_0)} \overline h_{M_j}\,\D\varphi \to \int_{\overline \omega(\alpha_0)} \overline h_{M}\,\D\varphi<\infty.$$
For $u\in\Omega_{C^\circ}\setminus \overline\omega(\alpha_0)$ we have $\delta_C(u)<\alpha_0$, hence $\overline h_{M_j}(u) \le c_1\delta_C(u)^{1/n}$ by Theorem \ref{T5.1}. This gives
$$ \int_{\Omega_{C^\circ}\setminus \overline\omega(\alpha_0)} \overline h_{M_j}\,\D\varphi \le c_1 \int_{\Omega_{C^\circ}\setminus \overline\omega(\alpha_0)} \delta_C^{1/n}\,\D\varphi<\infty$$
by (\ref{6.0}). Both estimates together yield (\ref{6.1}).

By (\ref{6.1}), the sequence $(\lambda_j)_{j\in\N}$ defined by (\ref{6.3}) is bounded. Since $V_n(C\setminus M_j)=1$, this implies the existence of a constant $c_6$ such that $V_n(C\setminus K_j)< c_6$. This, in turn, implies that the sequence $(K_j)_{j\in\N}$ has bounded distances from the origin, and these distances are also bounded away from $0$ since the sequence $(\lambda_j)_{j\in}$ is increasing. Hence, the sequence $(K_j)_{j\in\N}$ has a subsequence converging to a pseudo-cone $K$. After renumbering, we can assume that the sequence $(K_j)_{j\in\N}$ itself converges to $K$. By Lemma \ref{L2.3}, $K$ is $C$-close.

Let $k\in\N$. By Lemma \ref{L5.1}, there is a number $t_k$ such that 
$${\bm x}_{K_j}(\omega_k) \subset C^-(t_k)\quad\mbox{for } j\ge k.$$
For $j\ge k$, the restrictions to $\omega_k$ satisfy
$$ \varphi\fed\omega_k= \varphi_j\fed\omega_k= S_{n-1}(K_j,\cdot)\fed\omega_k= S_{n-1}(K_j\cap C^-(t_k),\cdot)\fed\omega_k.$$
Since $K_j\cap C^-(t_k)\to K\cap C^-(t_k)$, we have
$$ S_{n-1}(K_j\cap C^-(t_k),\cdot)\fed\omega_k \to  S_{n-1}(K\cap C^-(t_k),\cdot)\fed\omega_k \quad\mbox{weakly}$$
and hence 
$$ \varphi\fed\omega_k= S_{n-1}(K\cap C^-(t_k),\cdot)\fed\omega_k= S_{n-1}(K,\cdot)\fed\omega_k.$$
Since $k\in\N$ was arbitrary and $\bigcup_{k\in\N}\omega_k=\Omega_{C^\circ}$, we deduce that $\varphi= S_{n-1}(K,\cdot)$. This completes the proof. \hfill$\Box$

\vspace{3mm}

\noindent{\bf Acknowledgment.} We thank Jiming Zhao for pointing out an error in a former version.

\noindent Author's address:\\[2mm]
Rolf Schneider\\Mathematisches Institut, Albert--Ludwigs-Universit{\"a}t\\D-79104 Freiburg i.~Br., Germany\\E-mail: rolf.schneider@math.uni-freiburg.de

\end{document}